\documentclass[12pt,a4paper]{article}
\usepackage{a4wide}
\usepackage{amsmath,amsthm,amssymb}
\usepackage{array,enumerate}
\usepackage{indentfirst}	%paragraph
\usepackage{url,authblk,color,cite}	%address, author and color

\usepackage{hyperref}
\usepackage{xurl}

\newtheorem{theorem}{Theorem}[section]

\newtheorem{lemma}[theorem]{Lemma}

\newtheorem*{claim}{\textbf{Claim}}

\usepackage{tikz}
\usetikzlibrary{shapes}

\title{On the disjunctive domination numbers of the torus grid graphs}

\author[a]{Zhi Qiao}
\author[b]{Zheng-Jiang Xia 
	{\thanks{Corresponding author. \protect \\
	\indent \ \ 
	E-mail: 
    zhiqiao@sicnu.edu.cn (Z. Qiao),
    xzj@mail.ustc.edu.cn (Z.-J. Xia), zmhong@mail.ustc.edu.cn (Z.-M. Hong) }}}
\author[c]{Zhen-Mu Hong}
\affil[a]{ School of Mathematical Sciences, Sichuan Normal University, Chengdu, 610068, China}
\affil[b]{ School of Finance, Anhui University of Finance \& Economics, Bengbu, 233030, China}
\affil[c]{School of Statistics and Applied Mathematics, Anhui University of Finance \& Economics, Bengbu 233030, China}

\date{}

\begin{document}
\maketitle

\begin{abstract}
	Let $\Gamma=(V,E)$ be a graph. 
	The disjunctive domination number of $\Gamma$ is the minimum cardinality of a set $S\subseteq V$ such that every vertex not in $S$ is adjacent to a vertex of $S$, or has at least two vertices in $S$ at distance $2$ from it. 
	In this paper, we give bounds for the disjunctive domination numbers of the torus grid graphs $C_m\Box C_n$, and determine the disjunctive domination numbers of  $C_3\Box C_n$, $C_4\Box C_{n}$ and $C_8\Box C_{4n}$. 
	
	{\bf Keywords}: Disjunctive domination, Torus grid graphs
\end{abstract}

\section{Introduction}
	Let $\Gamma=(V(\Gamma),E(\Gamma))$ be a simple graph. 
	The distance between $u,v\in V(\Gamma)$ in the graph $\Gamma$ is denoted by $d_\Gamma(u,v)$. We use $d(u,v)$ if the graph is clear. 
	We set $\Gamma_i(v)=\{u \in V(\Gamma) \mid d(u,v)=i\}$, with $\Gamma(v):=\Gamma_1(v)$. 
	The {\em  Cartesian product} $G \Box H$ of two graphs $G$ and $H$ is a graph with vertex set $V(G) \times V(H)$, where two vertices $(g,h),(g',h')\in V(G\Box H)$ are adjacent if and only if either $g=g'$ and $hh' \in E(H)$, or $h=h'$ and $gg' \in E(G)$. The graphs $C_m\Box C_n$ are called {\em torus grid graphs}. 
	
	Goddard et al. \cite{G14} introduced the concept of disjunctive dominating set in a graph. 
	A set $S$ of $V(\Gamma)$ is a called a {\em disjunctive dominating set} (abbreviated 2DD-set) if for every vertex $v \in V(\Gamma) \setminus S$, either $|\Gamma(v)\cap S|\geq 1$, or $|\Gamma_2(v)\cap S|\geq 2$. 
	The minimum cardinality of a 2DD-set in $\Gamma$ is called the {\em disjunctive domination number}, denoted by $\gamma_2^d(\Gamma)$. 
	Since its introduction, it has been studied in \cite{H15,H16,H23,J20,P18,Y26,Y26a,Z22,Z23}. 
	
	The disjunctive domination numbers are determined for $P_i\Box P_n$ and $P_i\Box C_n$ with $i=2,3$, see \cite{G14,Y26a,Y26}. 
	In this paper, we give bounds for the disjunctive domination numbers of the torus grid graphs $C_m\Box C_n$, and determine the disjunctive domination numbers of $C_3\Box C_n$, $C_4\Box C_n$ and $C_8\Box C_{4n}$. 
	
	\begin{theorem}\label{thm:bound}
		For the torus grid graph $C_m \Box C_n$ ($m,n \ge 5$), the disjunctive domination number satisfies
		$$\frac{mn}{9}\leq \gamma_2^d(C_m \Box C_n)\leq 2\left\lceil \frac{m}{4}\right\rceil \cdot \left\lceil \frac{n}{4}\right\rceil.$$
	\end{theorem}
		
	\begin{theorem}\label{thm:3}
		For the torus grid graph $C_3 \Box C_n$ ($n \ge 3$), the disjunctive domination number is
		$$\gamma_2^d(C_3 \Box C_n) = \left\lceil \frac{n}{2} \right\rceil.$$
	\end{theorem}
	
	\begin{theorem}\label{thm:4}
		For the torus grid graph $C_4 \Box C_n$ ($n \ge 4$), the disjunctive domination number is
		$$
		\gamma_2^d(C_4 \Box C_n) =\left\{
		\begin{aligned}
			 & \frac{n}{2} + 1,                      &  & n \equiv 2 \pmod 4, \\
			 & \left\lceil \frac{n}{2} \right\rceil, &  & \text{otherwise}.
		\end{aligned}\right.
		$$
	\end{theorem}
	
	\begin{theorem}\label{thm:8}
		For the torus grid graph $C_8 \Box C_{n}$ $(n\geq 8)$, the disjunctive domination number satisfies
		$$
		n\leq \gamma_2^d(C_8 \Box C_{n})\leq \left\{
		\begin{aligned}
			& n+1, &  & n\equiv 1 \text{ or } 3 \pmod 4,\\
            & n+2,&  & n \equiv 2 \pmod 4. 
		\end{aligned}\right.
		$$ 
		If $n\equiv 0\pmod 4$, then $\gamma_2^d(C_8 \Box C_{n}) =n$. 
	\end{theorem}
	
\section{Preliminaries}
		We denote $[k]=\{0,\ldots, k-1\}$ and $V(C_k)=[k]$, where two vertices $i,j\in V(C_k)$ are adjacent if and only if $i-j\equiv \pm1 \pmod k$. 
		
		In this paper, we only consider the graph $\Gamma=C_m\Box C_n$ with $V(\Gamma)=[m]\times [n]$. 
		
		Let $S$ be a 2DD-set of $\Gamma$. 		
		Partition the vertex set of $C_m \Box C_n$ into $n$ columns 
		$$T_j = \{ (i,j)\mid i\in [m] \},\qquad j=0,1,\ldots, n-1.$$
		Let $x_j = |S \cap T_j|$, we call 
		$x_0,x_1,\ldots, x_{n-1}$
		the {\em sequence} of the 2DD-set $S$. 
		If there exists some $x_\ell=0$, by symmetry, we can always let $x_0>0$ and $x_{n-1}=0$. 
        In this case, such a sequence is called {\em proper}. 
		
		For a proper sequence $x_0,x_1,\ldots, x_{n-1}$ of $S$, partition it into blocks $X_0,\ldots,X_{e-1}$ with $$X_i=(x_{k_i+1},\ldots,x_{k_i+p_i},x_{k_i+p_i+1},\ldots, x_{k_i+p_i+q_i}),$$ 
		satisfying  
		$$x_{k_i+1},\ldots,x_{k_i+p_i}>0,\qquad x_{k_i+p_i+1}=\cdots=x_{k_i+p_i+q_i}=0.$$
		We call $X_0,\ldots, X_{e-1}$ the {\em block sequence} of $S$. 
		For fixed $a,b\in \mathbb N$, with $v_\ell:=a x_\ell -b$, define 
		$$V_i=v_{k_i+1}+v_{k_i+2}+\cdots+v_{k_i+p_i+q_i-1}+v_{k_i+p_i+q_i},$$
		the sequence $V_0,\ldots, V_{e-1}$ is called the {\em $(a,b)-$block sequence} of $S$. 		
		For simplicity, we extend the above notations for all $\ell\in \mathbb Z$ with
		$$\begin{aligned}
			 & T_\ell:=T_j, &  & x_\ell:=x_j, & &v_\ell:=v_j,&  & \text{if $\ell \equiv j \pmod n$,}  \\
			 & X_\ell:=X_i, &  & V_\ell:=V_i, & && & \text{if $\ell \equiv j \pmod e$.}
		\end{aligned}$$

		We define the weight function
		\begin{equation}\label{eq:weight}
			w:S\times V(\Gamma)\to \mathbb R,\qquad w(x,y)=\left\{
			\begin{aligned}
				& 1           &  & x=y,                             \\
				& 1           &  & y\notin S \text{ and } d(x,y)=1, \\
				& \frac{1}{2} &  & y\notin S \text{ and } d(x,y)=2, \\
				& 0           &  & \text{otherwise}.
			\end{aligned}	
			\right.
		\end{equation}

		\begin{lemma}\label{lm:sequence}
			Let $\Gamma=C_m\Box C_n$ $(m\geq 3,n\geq 5)$, and $S$ is a 2DD-set of $\Gamma$ with sequence $x_0,\ldots, x_{n-1}$. Then
			\begin{equation}\label{eq:basic}
				x_{j-2}+4x_{j-1}+8x_j+4x_{j+1}+x_{j+2} \geq 2m,\qquad j=0,1,\ldots, n-1. 
			\end{equation}
%			Moreover, if $x_j=0$, then 
%			\begin{equation}
%				x_{j-2}+4x_{j-1}+4x_{j+1}+x_{j+2} \geq 2m.
%			\end{equation}
		\end{lemma}
		\begin{proof}
			Let $S_\ell=S\cap T_\ell$. 
			Note that each vertex in $S_j$ has two vertices at distance 1 in $T_j$, and at most two vertices at distance 2 in $T_j$. 
			Each vertex in $S_{j-2}\cup S_{j+2}$ has one vertex at distance 2 in $T_j$, and each vertex in $S_{j-1}\cup S_{j+1}$ has one vertex at distance 1 and two vertices at distance 2 in $T_j$. 
			Consider the weight function \eqref{eq:weight}. We have
			$$
			\sum_{y\in T_j} w(x,y)\leq \left\{
			\begin{aligned}
				& 4,           &  & x\in S_j,                 \\
				& 2,           &  & x\in S_{j-1}\cup S_{j+1}, \\
				& \frac{1}{2}, &  & x\in S_{j-2}\cup S_{j+2}.
			\end{aligned}
			\right.
			$$
			Let $\Delta=S_{j-2}\cup S_{j-1}\cup S_j\cup S_{j+1}\cup S_j$. 
			Since $S$ is a 2DD-set, for any $y\in T_j$, we have $\Gamma(y)\cup \Gamma_2(y)\subseteq \Delta$ and 
			\begin{equation}\label{eq:xy}
				\sum_{x\in\Delta}w(x,y)\geq 1.
			\end{equation}
			Sum \eqref{eq:xy} for $y\in T_j$, we obtain
			$$
			\begin{aligned}
				m & \leq \sum_{y\in T_j}\left(\sum_{x\in \Delta} w(x,y)\right)               \\
				  & = \sum_{x\in\Delta}\left( \sum_{y\in T_j} w(x,y)\right)                  \\
				  & \leq\frac{1}{2}x_{j-2} + 2x_{j-1} + 4x_j + 2x_{j+1}+ \frac{1}{2}x_{j+2},
			\end{aligned}
			$$
			This shows the lemma. 
		\end{proof}
		
		\begin{lemma}\label{lm:q}
			Let $\Gamma=C_m\Box C_n$ $(m\geq 3,n\geq 5)$, and $S$ is a 2DD-set of $\Gamma$ with block sequence $X_0,\ldots, X_{e-1}$. Then
			$$q_i\leq 3,\qquad i=0,\ldots, e-1,$$
			where $q_i$ is the number of zeros in $X_i$. 
		\end{lemma}
		\begin{proof}
			If $q_i\geq 4$, then we have four consecutive zeros in the block $X_i$.  Assume $$(x_{\ell-2},x_{\ell-2},x_{\ell},x_{\ell+1})=(0,0,0,0). $$
			With $j=\ell$ in \eqref{eq:basic}, we have $x_{j+2}\geq 2m$, contradict to $x_j\leq m$. Hence $q_i\leq 3$. 
		\end{proof}
		
		\begin{lemma}\label{lm:reduce}
			Let $\Gamma=C_m\Box C_n$ $(m,n\geq 6)$, then
			$$
				\max \{\gamma_2^d(C_{m-1}\Box C_n),~\gamma_2^d(C_m\Box C_{n-1})\}\leq \gamma_2^d(C_m\Box C_n). 
			$$
		\end{lemma}
		\begin{proof}
			Let $\Gamma=C_m\Box C_n$ and $\Gamma'=C_{m}\Box C_{n-1}$. 
			Consider the following map 
			$$
			f:V(\Gamma)\to V(\Gamma'),\qquad f(v)=\left\{
			\begin{aligned}
				 & (x,y),   &  & \text{if $y\leq n-2$,} \\
				 & (x,y-1), &  & \text{if $y=n-1$},
			\end{aligned}\right.
			$$
			where $v=(x,y)$ is a vertex of $\Gamma$. 
			In fact, the map $f$ just folds the columns $T_{n-1}$ and $T_{n-2}$ of $\Gamma$ together. 
			Consider any two vertices $u,v\in V(\Gamma)$. 
			If $u,v$ are in the same column, then $f$ keeps their distance. 
			Otherwise, the distance is reduced by at most 1. 
			We have
			\begin{equation}\label{eq:distance}
				d_\Gamma (u,v)\geq d_{\Gamma'}(f(u),f(v))\geq d_\Gamma(u,v)-1. 
			\end{equation} 
			
			Let $S$ be a 2DD-set of $\Gamma$ with $|S|=\gamma_2^d(C_m\Box C_n)$. Consider the set $f(S)=\{f(u)\mid u\in S\}$. We see that $|f(S)|\leq |S|$. 
			Let $w=f(u)$ be any vertex in $V(\Gamma')\setminus f(S)$. 
			Note that $u\notin S$. 
			
			If $|\Gamma_1(u)\cap S|\geq 1$, take $v\in \Gamma_1(u)\cap S$.  
			Note that $w\neq f(v)\in f(S)$. By \eqref{eq:distance}, we have $d_\Gamma(f(v),w)= 1$.
			Then $f(v)\in \Gamma_1'(w)\cap f(S)$, which implies $|\Gamma_1'(w)\cap f(S)|\geq 1$. 		
				
			In the case $|\Gamma_2(u)\cap S|\geq 2$, take $v_1,v_2\in \Gamma_2(u)\cap S$, then $1\leq d_{\Gamma'}(f(v_i),w)\leq 2$ by \eqref{eq:distance}. 
			If $f(v_i)\in \Gamma_1'(w)\cap f(S)$ for some $i$, then $|\Gamma_1'(w)\cap f(S)|\geq 1$. 
			Otherwise,  we have $f(v_1),f(v_2)\in \Gamma_2'(w)\cap f(S)$. 
			As $m,n\geq 6$, we see that $d_\Gamma (v_1, v_2)\geq 2$, which implies $d_{\Gamma'}(f(v_1),f(v_2))\geq 1$ by \eqref{eq:distance}. 
			It follows that $|\Gamma_2'(w)\cap f(S)|\geq 2$. 
			
			Hence $f(S)$ is a 2DD-set of $\Gamma'$. 
			This shows 
			$$
			\gamma_2^d(C_{m}\Box C_{n-1})\leq \gamma_2^d(C_m\Box C_n). 
			$$
			
			By symmetry, the result for $C_{m-1}\Box C_{n}$ follows.

		\end{proof}

\section{Proof of Theorem \ref{thm:bound}}
\begin{proof}[Proof of Theorem \ref{thm:bound}]
	Let $\Gamma=C_m\Box C_n$ and $S$ be a 2DD-set of $\Gamma$. 
	Consider the weight function \eqref{eq:weight}. 
	For any vertex $y\in V(\Gamma)\setminus S$, we have $|\Gamma (y)\cap S|\geq 1$, or $|\Gamma_2(y)\cap S|\geq 2$, that is
	$$\sum_{x\in S}w(x,y)\geq 1.$$
	For any vertex $x\in S$, as $m,n\geq 5$, we have $|\Gamma(x)|=4$ and $|\Gamma_2(x)|=8$, which implies
	$$\sum_{y\in V(\Gamma)\setminus S}w(x,y)\leq 4+8\cdot\frac{1}{2}=8.$$
	We have  
	$$
	\begin{aligned}
		     & |V(\Gamma)|-|S|                                                      \\
		\leq & \sum_{y\in  V(\Gamma)\setminus S}\left( \sum_{x\in S}w(x,y) \right) \\
		=    & \sum_{x\in S}\left(\sum_{y\in  V(\Gamma)\setminus S} w(x,y) \right) \\
		\leq & 8|S|.
	\end{aligned}$$
	It follows that $|S|\geq |V(\Gamma)|/9$ for any 2DD-set $S$, that is 
	\begin{equation}\label{eq:low}
		\gamma_2^d(C_m\Box C_n)\geq \frac{mn}{9}.
	\end{equation}
		
	Let $m'=4\lceil \frac{m}{4}\rceil$, $n'=4\lceil \frac{n}{4}\rceil$ and 
	$$S'=\{
		(x',y')\in [m']\times [n']\mid  \text{$x',y'$ are even, and }x'+y'\equiv 0\pmod 4
	\}.$$
	Then $S'$ is a 2DD-set of $\Gamma'=C_{m'}\Box C_{n'}$, and we have
	$$
	\gamma_2^d (C_{m'}\Box C_{n'})\leq |S'|=2\lceil \frac{m}{4}\rceil \cdot \lceil \frac{n}{4}\rceil. 
	$$	 
	If $m,n\geq 6$, by Lemma \ref{lm:reduce}, we have
	\begin{equation}\label{eq:up}
		\gamma_2^d (C_m\Box C_n)\leq 2\lceil \frac{m}{4}\rceil \cdot \lceil \frac{n}{4}\rceil. 
	\end{equation}
	
	Otherwise, we may assume $m=5$ by symmetry. Then $2\left\lceil \frac{m}{4}\right\rceil \cdot \left\lceil \frac{n}{4}\right\rceil\geq n$. 
	We construct a 2DD-set $S$ of $\Gamma=C_5\Box C_n$, with  
	$$S=\{(i,j)\in [5]\times [n]\mid i\equiv 2j \pmod 5\},$$
	and $|S|=n$, which implies \eqref{eq:up}. 
	
	The result follows from \eqref{eq:low} and \eqref{eq:up}.

\end{proof}

\section{Proof of Theorem \ref{thm:3}}
\begin{proof}[Proof of Theorem \ref{thm:3}]
	
	\begin{figure}%[htpb]
		\centering
		\includegraphics[width=0.8\textwidth]{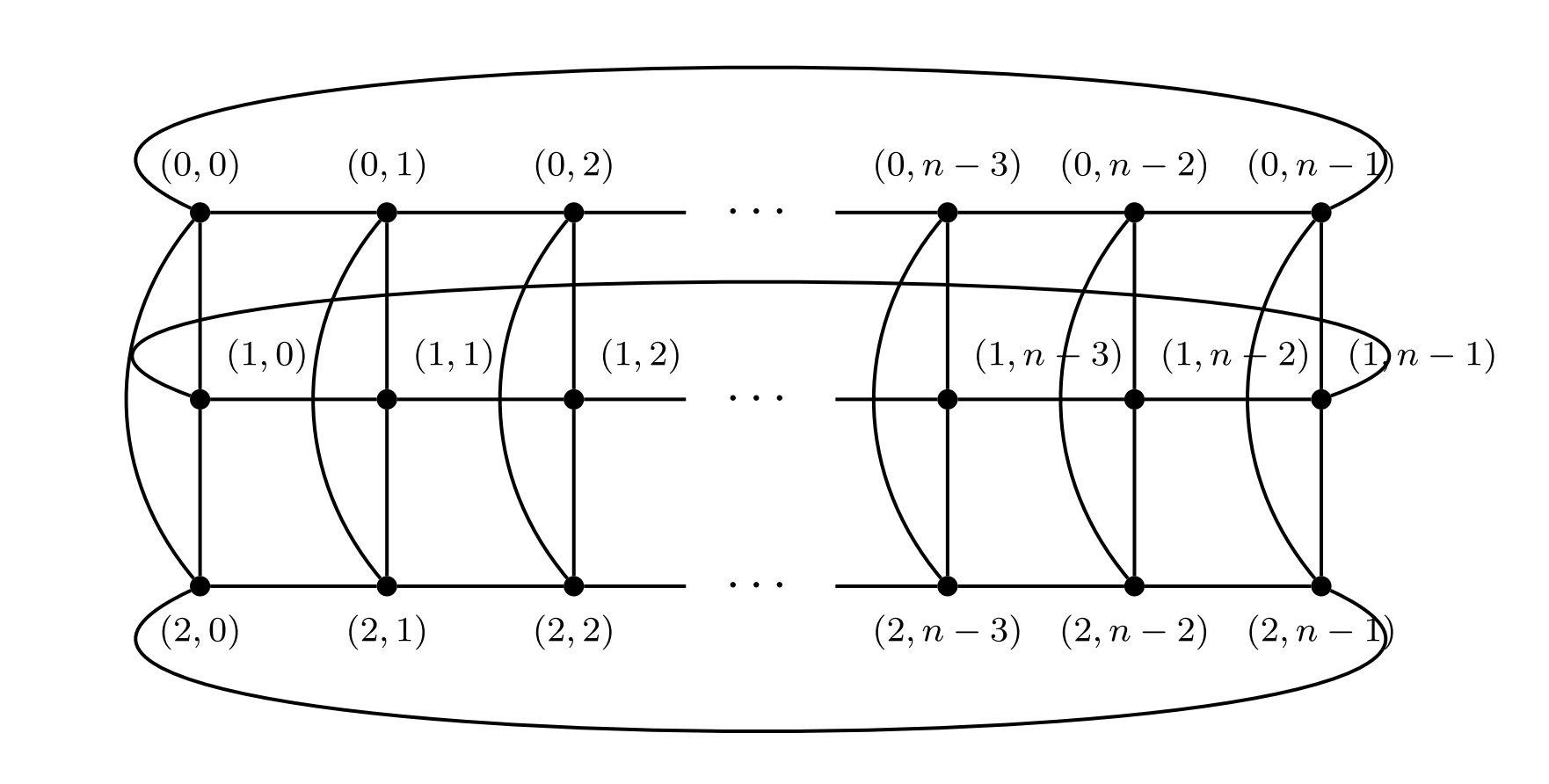}
	\caption{The torus grid graph $C_3 \Box C_n$.}
	\label{fig:C3_Cn}
\end{figure}
    For $n=3,4$, the result follows from direct verification. Assume $n \ge 5$.

	Let $\Gamma=C_3 \Box C_n$. 
	With $A=\{(0, 2k)\mid k\in \mathbb N\}$, we construct a 2DD-set $S_0$ of $\Gamma$, where $S_0=A\cap V(\Gamma)$ and $|S_0| = \lceil \frac{n}{2} \rceil$. This shows 
	\begin{equation}\label{eq:3up}
		\gamma_2^d(C_3 \Box C_n)\leq  \left\lceil \frac{n}{2} \right\rceil .
	\end{equation}	
	
    Let $S$ be a 2DD-set with sequence $x_0,\ldots, x_{n-1}$. 
	If all $x_{\ell}\geq 1$, then 
	$$
	|S|=\sum_{\ell=0}^{n-1} x_{\ell}\geq n.
	$$
	So we may assume the sequence is proper, with block sequence $X_0,\ldots, X_{e-1}$.  
    Let $V_0,\ldots, V_{e-1}$ be the $(2,1)$-block sequence  with $v_\ell=2x_\ell -1$. 
	Consider a block $X_i$, with $q_i$ zeros. Then $q_i\leq 3$ by Lemma \ref{lm:q}. 
		
	{\bf Case 1.} If $q_i=3$, let $x_\ell$ be the middle zero of $X_i$, that is $x_{\ell-1}=x_\ell=x_{\ell+1}=0$, by $\eqref{eq:basic}$ with $j=\ell$, 
	we have $x_{\ell-2}+x_{\ell+2}\geq 6$. 
	It follows that $x_{\ell-2}=3$ and
	\begin{equation}\label{case31}
		V_i\geq (2x_{\ell-2}-1)-q_i = 2.
	\end{equation}
	
	{\bf Case 2.} If $q_i=1$, we have $V_i \geq 1-q_i\ge 0$.   
	If some $x_\ell\geq 2$ in $X_i$, we  have
	\begin{equation}
		V_i \geq (2x_\ell-1) -q_i\geq 2. \label{case32}
	\end{equation}

	{\bf Case 3.} Consider the case $q_i=2$. If some $x_\ell\geq 2$ in $X_i$, then we have 
	\begin{equation}\label{case33}
		V_i\geq (2x_\ell-1)-q_i\geq 1.
	\end{equation}
	Assume all $x_\ell \leq 1$ in $X_i$. 
	Then $V_i < 0$ if and only if $X_i=(x_{\ell-1},x_{\ell},x_{\ell+1})=(1,0,0)$ and $V_i=-1$. 
	Substituting  $(x_{\ell-1},x_{\ell},x_{\ell+1})=(1,0,0)$ into $\eqref{eq:basic}$ with $j=\ell$, we have  
	$$x_{\ell-2} + x_{\ell+2} \ge 2.$$ 
	Note that $x_{\ell-2}$ is the last zero element of $X_{i-1}$, we have $x_{\ell+2}\geq 2$. 
	By \eqref{case31}, \eqref{case32} and \eqref{case33}, we have $V_{i+1}\geq 1$. 
	
	We see that if $V_i<0$, then $V_i+V_{i+1}\geq 0$. 
    Thus, 
	$$0\leq \sum_{i=0}^e V_i=\sum_{\ell=0}^{n-1} v_\ell=\sum_{\ell=0}^{n-1} (2x_\ell-1)=2|S| - n,$$
	that is 
	\begin{equation}\label{eq:3low}
		\gamma_2^d(C_3 \Box C_n) \geq \lceil n/2 \rceil.
	\end{equation}
	
	The result follows from \eqref{eq:3up} and \eqref{eq:3low}. 
\end{proof}

\section{Proof of Theorem \ref{thm:4}}
	\begin{figure}%[htpb]
		\centering
		\includegraphics[width=0.8\textwidth]{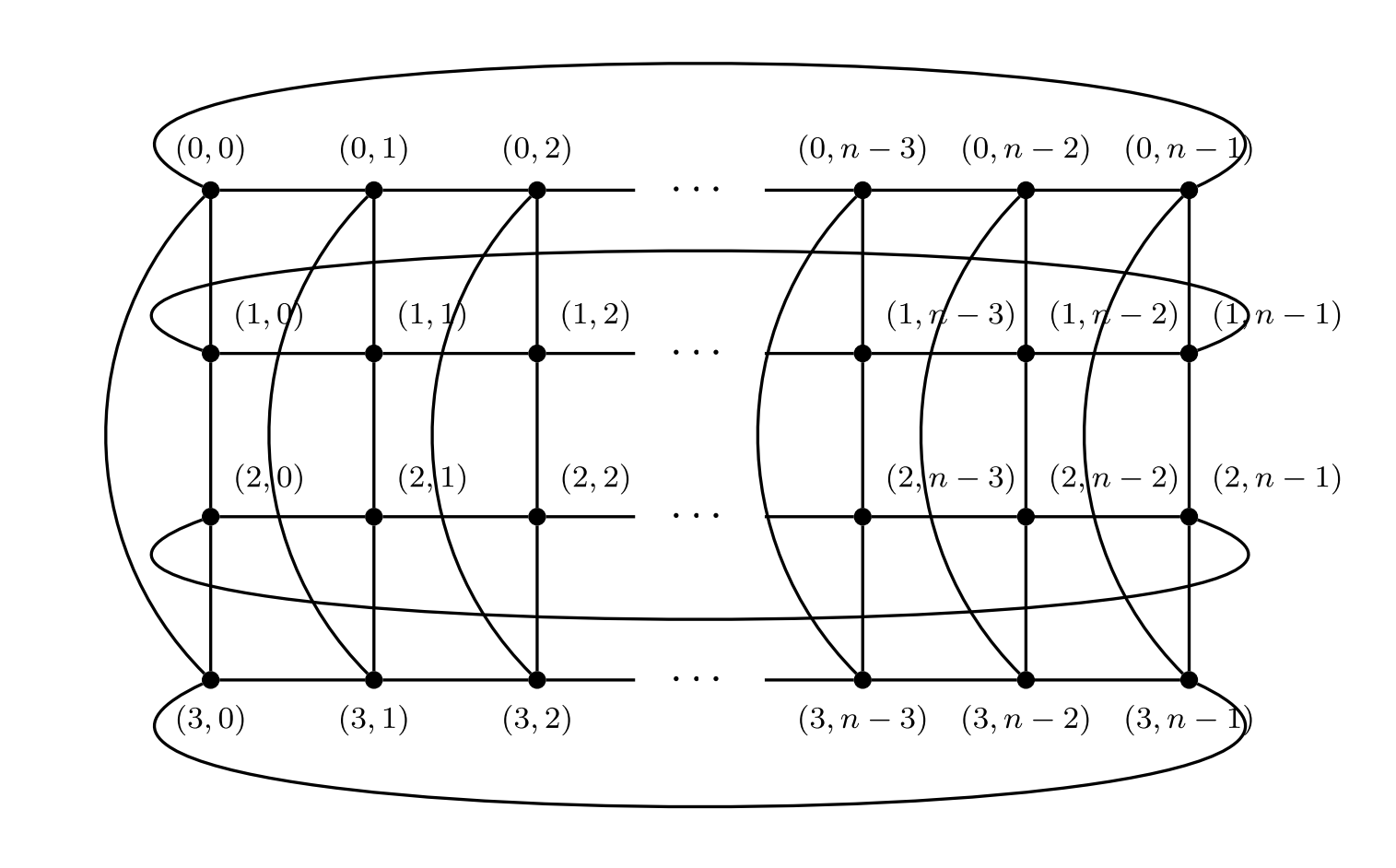}
		\caption{The torus grid graph $C_4 \Box C_n$.}
		\label{fig:C4_Cn}
	\end{figure}
	\begin{proof}[Proof of Theorem \ref{thm:4}]
		For $n=4$, direct verification shows the result. Assume $n\geq 5$. 
		
		Let $\Gamma=C_4\Box C_n$. 
		With $A= \{(0,4k),(2,4k+2)\mid k\in \mathbb N\}$ and $A_0=A\cap V(\Gamma)$,  
		we construct a 2DD-set $S_0$ of $\Gamma$, where 
		$$
		S_0=\left\{
			\begin{aligned}
				 & A_0 \cup \{(2,0)\}, &  & n \equiv 2 \pmod 4, \\
				 & A_0,                &  & \text{otherwise,}
			\end{aligned}
		\right.
		$$
		and 
		$$
		|S_0| =\left\{
		\begin{aligned}
			 & \frac{n}{2} + 1,                      &  & n \equiv 2 \pmod 4, \\
			 & \left\lceil \frac{n}{2} \right\rceil, &  & \text{otherwise}.
		\end{aligned}\right.
		$$
		This shows that
		\begin{equation}\label{eq:4up}
			\gamma_2^d(C_4\Box C_n)\leq |S_0|.
		\end{equation}

	Let $S$ be a 2DD-set of $\Gamma$ with sequence $x_0,\ldots, x_{n-1}$. 
	If $(x_\ell,x_{\ell+1},x_{\ell+2},x_{\ell+3})=0$, then  with $j=\ell+1$ in \eqref{eq:basic}, we have  $x_{\ell-1}\geq 8$, contradict to $x_{\ell-1} \leq 4$. Hence we have 
	\begin{equation*}
		\sum_{t=0}^3 x_{\ell+t}\geq 1, \qquad \ell=0,1,\ldots, n-1.  
	\end{equation*}
	Assume equality holds for some $\ell=k$. 
	Then there exists only one vertex $v$ of $S$ in the four columns $T_k,\ldots, T_{k+3}$. 
	By symmetry, we may assume $k=0$ and $u=(0,y)$ with $y\in\{0,1\}$. Consider the vertex $v=(2,2)$. We see that $d(u,v)\geq 3$ and $\Gamma_2(v)\cap S\subset \{(2,4)\}$, contradict to $|\Gamma_2(v)\cap S|\geq 2$. Hence we have
	\begin{equation}\label{eq:41}
		\sum_{t=0}^3 x_{\ell+t}\geq 2, \qquad \ell=0,1,\ldots, n-1, 
	\end{equation}
	and  
	\begin{equation}\label{eq:42}
		4|S|=4\sum_{\ell=0}^{n-1} x_\ell=\sum_{\ell=0}^{n-1}\sum_{t=0}^3 x_{\ell+t}\geq 2n.
	\end{equation}
	that is
	\begin{equation}\label{eq:4low}
		|S|\geq \left\lceil \frac{n}{2}\right\rceil.
	\end{equation}
	
	\begin{claim}
		If $n \equiv 2 \pmod 4$, then $|S| \ge n/2 + 1$.
	\end{claim}
	\begin{proof}[\textbf{Proof of Claim.}] 
		If $|S| = n/2$, then by \eqref{eq:41} and \eqref{eq:42}, we see
		$$\sum_{t=0}^3 x_{\ell+t} = 2,\qquad \ell=0,1,\ldots,n-1.$$
		This implies $x_\ell = x_{\ell+4}$. Since $n \equiv 2 \pmod 4$ and $\gcd(4, n) = 2$, we see $x_\ell = x_{\ell+2}$. Then the sequence $x_\ell$ must be $1, 0, 1, 0, \dots$. 
			
		Assume $x_0=1$. Then the sequence of $S$ is $1,0,1,0,\ldots, 1,0$, where $S$ contains exactly one vertex $u_{2i}=(r_{2i}, 2i)$ in $T_{2i}$, and $0$ vertex in $T_{2i+1}$.  
		Consider $u=(x,y)$ with
		$$
		x\equiv r_{2i}+2 \pmod 4,\qquad y\equiv 2i+1\pmod n.
		$$ 
		We see that $d(u_{2i},u)=3$, and the vertices in $S\setminus \{u_{2i+2}\}$ are at distance at least $3$ from $u$. 
		So we must have $d(u_{2i+2},u)=1$, 
		which means
		$r_{2i+2}\equiv r_{2i}+2\pmod 4$. 
		We see that
		$$r_0\equiv r_{0}+2\cdot \frac{n}{2}\pmod 4,$$
		contradict to $n\equiv 2\pmod 4$. 
		This shows $|S|\geq \frac{n}{2}+1$. 
	\end{proof}
	The result follows from \eqref{eq:4up}, \eqref{eq:4low} and the Claim. 
\end{proof}

\section{Proof of Theorem \ref{thm:8}}

\begin{proof}[Proof of Theorem \ref{thm:8}]
	Consider the graph $\Gamma=C_8\Box C_n$. 
	Let $$A= \{(0,4k),(4,4k),(2,4k+2),(6,4k+2)\mid k\in \mathbb Z\},$$ 
    and $A_0=A\cap V(\Gamma)$. 
	We construct a 2DD-set $S_0$ of $\Gamma$, where
	$$
	S_0=\left\{
	\begin{aligned}
		 & A_0 \cup \{(2,0),(6,0)\}, &  & n \equiv 2 \pmod 4, \\
		 & A_0,                      &  & \text{otherwise,}
	\end{aligned}
	\right.
	$$
	and 
	$$
	|S_0| =\left\{
	\begin{aligned}
		& n+2,                                   &  & n \equiv 2 \pmod 4, \\
		& 2\left\lceil \frac{n}{2} \right\rceil, &  & \text{otherwise}.
	\end{aligned}\right.
	$$
	This shows that 
	\begin{equation}\label{eq:8up}
		\gamma_2^d(C_8\Box C_{n})\leq |S_0|.
	\end{equation}
	
	Let $S$ be a 2DD-set with sequence $x_0,\ldots, x_{n-1}$. 
	If all $x_{\ell}\geq 1$, then 
	$$
	|S|=\sum_{\ell=0}^{n-1} x_{\ell}\geq n.
	$$
	Assume the sequence is proper. 
	Let $X_0,\ldots, X_{e-1}$ be block sequence of $S$, and $V_0,\ldots, V_{e-1}$ be the $(1,1)$-block sequence  with $v_\ell=x_\ell -1$. 
	Consider a block $X_i$, with $q_i$ zeros. Then $q_i\leq 3$ by Lemma 
	\ref{lm:q}. Let $x_{\ell-1}$ be the first zero in $X_i$. 
	
	For any block $X_k$ with $q_k$ zeros, and $x_{s}$ a nonzero term in $X_k$, we have 
	\begin{equation}\label{eq:s}
		V_k\geq (x_{s}-1)-q_k. 
	\end{equation}
	Moreover, 	
	\begin{equation}\label{eq:s+1}
		V_k\geq (x_s-1)+(x_{s+1}-1)-q_k, 
	\end{equation}
    \begin{equation}\label{eq:s-1}
		V_k\geq (x_{s-1}-1)+(x_{s}-1)-q_k.  
	\end{equation}

	{\bf Case 1.}
	If $q_i=3$, then the zeros in $X_i$ are $x_{\ell-1},x_{\ell},x_{\ell+1}$. 
	With $j=\ell$ in \eqref{eq:basic}, we see $x_{\ell-2}+x_{\ell+2}\geq 16$ and $x_{\ell-2}=8$. 
    By \eqref{eq:s}, we have $V_i\geq (x_{\ell-2}-1)-q_i\geq 4$. 
	
	{\bf Case 2.} 
	If $q_i=2$, then the zeros in $X_i$ are $x_{\ell-1},x_\ell$. Substitute $j=\ell,\ell-1$ in \eqref{eq:basic} respectively, we obtain 
	\begin{equation}\label{eq:+1}
		4x_{\ell+1}+x_{\ell+2}  \geq 16-x_{\ell-2},
	\end{equation}
	\begin{equation}\label{eq:-3}
		x_{\ell-3}+x_{\ell+1}  \geq 16-4x_{\ell-2}.
	\end{equation}
	We have $V_i\geq (x_{\ell-2}-1)-2$ by \eqref{eq:s}
    We see that if $V_i< 0$, then $x_{\ell-2}\leq 2$. 
	
	{\bf Case 2a.} Consider the case $x_{\ell-2}=1$. 
	With $x_{\ell+1}\leq 8$, by \eqref{eq:-3}, we have 
	$x_{\ell-3}\geq 12-x_{\ell+1}\geq 4$. Then \eqref{eq:s} implies  $V_i\geq (x_{\ell-3}-1)-2\geq 1$. 
	It follows that the block
	$$
	X_t=(3,1,\ldots,1,0,0)
	$$
	is impossible. 
	
	{\bf Case 2b.} Consider the case $x_{\ell-2}=2$ and $V_i< 0$. 
	Then by \eqref{eq:s} and \eqref{eq:s-1}, we have
	$$
	-1\geq V_i\geq (x_{\ell-2}-1)-2=-1,
	$$
	$$-1\geq V_i\geq (x_{\ell-3}-1)+(x_{\ell-2}-1)-2=x_{\ell-3}-2.
	$$
	It follows that $x_{\ell-3}\leq 1$ and $V_i=-1$. And \eqref{eq:-3} implies $x_{\ell+1}\geq 7$. 
	Then we obtain $V_{i+1}\geq (x_{\ell+1}-1)-3\geq 1$ by \eqref{eq:s}, and  $V_i+V_{i+1}\geq 0$.

    In general, if $V_k\leq 0$, then we have $q_k\leq 2$ and \eqref{eq:s} and \eqref{eq:s+1} implies
	\begin{equation}\label{eq:34}
		x_s\leq 3,\qquad x_s+s_{s+1}\leq 4. 
	\end{equation}

	{\bf Case 2c.} Consider the case $x_{\ell-2}=3$ and $V_i=0$. 
	By \eqref{eq:+1}, we have $4x_{\ell+1}+x_{\ell+2}\geq 13$. 
	If $V_{i+1}\leq 0$, then we have \eqref{eq:34} with $s=\ell+1$. 
    Since the only solution of
	$$\left\{
	\begin{aligned}
		&4x_{\ell+1}+x_{\ell+2}\geq 13\\
		&x_{\ell+1}\leq 3\\
		&x_{\ell+1}+x_{\ell+2}\leq 4
	\end{aligned}\right. 
	$$
	is $(x_{\ell+1},x_{\ell+2})=(3,1)$, we must have
	$$X_{i+1}=(3,1,\ldots,1,0,0)$$
	to make $V_{i+1}\leq 0$, contract to Case 2a. Hence we always have $V_{i+1}\geq 1$. 
	
	{\bf Case 3.}
	Consider the case $q_i=1$ and $V_i<0$.
	The zero in $X_i$ is exactly $x_{\ell-1}$. 
	Note that $x_{\ell-2}\geq 1$, and \eqref{eq:s} and \eqref{eq:s-1} implies
	$$-1\geq V_i\geq x_{\ell-2}-2,$$
	$$-1\geq V_i\geq x_{\ell-3}+x_{\ell-2}-3.$$
	We see that  $V_i=-1$,  and $(x_{\ell-3},x_{\ell-2})=(0,1)$ or $(1,1)$. 
	We have 
	$4x_{\ell}+x_{\ell+1}\geq 11$ with $j=\ell-1$ in \eqref{eq:basic}. 
	If $V_{i+1}\leq 0$, then we have \eqref{eq:34} with $s=\ell$. 
    Since the only solution of
	$$
	\left\{
	\begin{aligned}
		&4x_{\ell}+x_{\ell+1}\geq 11\\
		&x_{\ell}\leq 3\\
		&x_{\ell}+x_{\ell+1}\leq 4
	\end{aligned}\right. 
	$$
	is $(x_{\ell+1},x_{\ell+2})=(3,0)$ or $(3,1)$. 
	Then $V_{i+1}\leq 0$ implies $V_{i+1}=0$, with
	$X_{i+1}=(3,0,0)$ or $(3,1,\ldots,1,0,0)$. 
	The latter one is impossible by Case 2a. 
	For the former one, Case 2c implies $V_{i+2}\geq 1$. 
	We have $V_{i+1}=0$ and $V_i+V_{i+2}\geq 0$.  

    We show that if $V_i<0$ then one of the following holds
    \begin{enumerate}[1)]
        \item $V_i+V_{i+1}\geq 0$ (Case 2b);
        \item $V_{i+1}=0$ and $V_i+V_{i+2}\geq 0$ (Case 3). 
    \end{enumerate}
	It follows that  
	$$
	0\leq \sum_{i=0}^{e-1} V_i=\sum_{\ell=0}^{n-1}x_{\ell}-n,
	$$
	that is
	\begin{equation}\label{eq:8low}
		\gamma_2^d(C_8\Box C_{n})\geq n.
	\end{equation}
	
	The result follows from \eqref{eq:8up} and \eqref{eq:8low}.
	
\end{proof}

\section*{ Data   availability}

No data was used for the research described in the article.
	
\section*{Acknowledgments}

Qiao is supported by Natural Science Foundation of China (No. 12071321). Xia is supported by Key Project in Natural Science Research of Anhui Provincial Department of Education (No. 2023AH050268). Hong is supported by Natural Science Foundation of China (No. 12371338) and Outstanding Youth Scientific Research Projects of Anhui Provincial Department of Education (No. 2022AH030073).

\bibliographystyle{plain}
\bibliography{bib}
\end{document}